\documentclass[
twoside]
{article}
\usepackage[a5paper,layoutwidth=148mm,layoutheight=7.77817in,outer=10mm,inner=19bp,tmargin=10mm,bmargin=10mm,footskip=15bp,]{geometry}
\usepackage{amsmath,amssymb,amsthm}
\usepackage{mathtools}
\mathtoolsset{mathic,centercolon}

\usepackage[T1]{fontenc}
\usepackage{lmodern}
\usepackage{microtype}
\usepackage{xspace}

\usepackage{natbib}
\bibpunct{(}{)}{,}{a}{}{,}

\usepackage[scr]{rsfso}

\usepackage[english,german,french,british]{babel}

\usepackage[bookmarksopen=false,breaklinks=true,%
backref=page,pagebackref=true,plainpages=false,%
hyperindex=true,pdfstartview=FitH,%
pdfpagelabels=true,linkcolor=blue,%
citecolor=red,urlcolor=blue,colorlinks=false,%
]%
{hyperref}
\usepackage{doi}
\renewcommand*{\backref}[1]{}
\renewcommand*{\backrefalt}[4]{\ifcase #1\or(p.~#2.)\else(pp.~#2.)\fi}
\renewcommand*{\backreftwosep}[1]{ and }
\renewcommand*{\backreflastsep}[1]{, and }

\newcommand{\coh}{coherent\xspace}
\newcommand{\Sdcc}{Strongly discrete coherent context\xspace}
\newcommand{\Sdic}{Simple division context\xspace}

\newtheorem{theorem}{Theorem}[section]
\newtheorem{maintheorem}[theorem]{Theorem}
\newtheorem{proposition}[theorem]{Proposition}

\newtheorem{corollary}[theorem]{Corollary}

\theoremstyle{remark}
\newtheorem{comment}[theorem]{Comment}
\newtheorem{remark}[theorem]{Remark}
\newtheorem{example}[theorem]{Example}

\theoremstyle{definition}
\newtheorem{definition}[theorem]{Definition}
\newtheorem{gencontext}[theorem]{General context}
\newtheorem{discohcontext}[theorem]{Discrete coherent context}
\newtheorem{sdccontext}[theorem]{Strongly discrete coherent context}
\newtheorem{simpledivicontext}[theorem]{Simple division context}

\newtheorem{notation}[theorem]{Notation}

\newtheorem{divalgorithm}[theorem]{Division algorithm}

\newtheorem{slistalgorithm}[theorem]{S-list algorithm}
\newtheorem{buchbalgorithm}[theorem]{Buchberger's algorithm}
\newtheorem{syzalgorithm}[theorem]{Syzygy algorithm for terms}
\newtheorem{basicsyzalgorithm}[theorem]{Basic syzygy algorithm for terms}
\newtheorem{rewritealgorithm}[theorem]{Rewriting algorithm}
\newtheorem{Syzalgorithm}[theorem]{Schreyer's syzygy algorithm}

\newcommand{\so}[1]{\left\{#1\right\}}
\newcommand{\sotq}[2]{\{\,#1\mathrel;\allowbreak#2\,\}}

\newcommand \som {\sum\nolimits}

\newcommand \St{\mathscr{S}}

\newcommand \n{\noindent}
\DeclareMathOperator{\LT}{LT}
\DeclareMathOperator{\MLT}{MLT}
\DeclareMathOperator{\LM}{LM}
\DeclareMathOperator{\LC}{LC}
\DeclareMathOperator{\LP}{LPos}

\DeclareMathOperator{\Syz}{Syz}
\DeclareMathOperator{\List}{List}
\DeclareMathOperator{\mdeg}{mdeg}
\DeclareMathOperator{\tdeg}{tdeg}
\DeclareMathOperator{\Ann}{Ann}

\DeclareMathOperator{\lcm}{lcm}

\newcommand\aLM{L}

\newcommand{\rs}{\mathrm{s}}
\newcommand \NN{\mathbb{N}}
\newcommand \ZZ{\mathbb{Z}}
\newcommand \QQ{\mathbb{Q}}

\newcommand \uX  {\underline{X}}
\newcommand \RuX  {\R[\uX]}
\newcommand \ug  {\underline{g}}
\newcommand \ugprime  {\underline{g'}}
\newcommand \ugseconde  {\underline{g''}}

\newcommand{\Pp}{\mathscr{P}_{\!p}}
\newcommand{\Pap}{\mathscr{P}}

\newcommand \eoe {\hbox{}\nobreak\hfill
  \vrule width .5em height .5em depth 0mm \par \smallskip}

\newcommand \lst[1]{#1}

\newcommand \gen[1]{\langle{#1}\rangle}
\newcommand \geN[1]{\big\langle{#1}\big\rangle}
\newcommand \gentq[2]{\gen{\,#1\mathrel;#2\,}}

\newcommand \F {\mathbf{F}}

\newcommand \R {\mathbf{R}}
\newcommand \V {\mathbf{V}}
\newcommand \Hn {\mathbf{H}_n}
\newcommand \Hm {\Hn^m}

\newcommand \noi {\noindent}
\newcommand \sms {\smallskip}
\newcommand \sni {\sms\noi}

\newcommand \bs {\bigskip}
\newcommand \bni {\bs\noi}

\makeatletter
\DeclareFontEncoding{LS2}{}{\noaccents@}
\makeatother

\DeclareFontSubstitution{LS2}{stix}{m}{n}
\DeclareSymbolFont{stix-largesymbols}  {LS2}{stixex}   {m} {n}

\DeclareMathDelimiter{\stixlBrack} {\mathopen} {stix-largesymbols}{"E0}{stix-largesymbols}{"06}
\DeclareMathDelimiter{\stixrBrack} {\mathclose}{stix-largesymbols}{"E1}{stix-largesymbols}{"07}

\newcommand\lrb[1]{[ #1 ]}
\newcommand{\twodots}{\mathinner {\ldotp \ldotp}}

\providecommand{\divides}{\mathbin|}

\usepackage{enumitem}
\setlist[itemize, 1]{wide}
\setlist[enumerate, 1]{wide}
\setlist[enumerate, 2]{wide,leftmargin = \parindent}%

\usepackage{listings}
\usepackage{authblk}

\newcommand{\X}[1][1]{\ifcase#1\or X\or Y\fi}

\begin{document}

\tolerance 1414
\hbadness 1414
\emergencystretch 1.5em
\hfuzz 0.3pt
\widowpenalty=10000
\vfuzz \hfuzz
\raggedbottom

\title{Generalised Buchberger and  Schreyer algorithms for strongly discrete coherent rings}
\author[*]{Henri Lombardi}
\author[*]{Stefan Neuwirth}
\author[**]{Ihsen Yengui}
\affil[*]{Université Marie et Louis Pasteur, CNRS, UMR 6623, LmB, 25000 Besançon, France, \href{mailto:henri.lombardi@umlp.fr}{henri.lombardi@umlp.fr}, \href{mailto:stefan.neuwirth@umlp.fr}{stefan.neuwirth@umlp.fr}.}
\affil[**]{Département de mathématiques, Faculté des sciences, Université de Sfax, 3000 Sfax, Tunisia, \href{mailto:ihsen.yengui@fss.rnu.tn}{ihsen.yengui@fss.rnu.tn}.}
\date{}
\maketitle

\begin{abstract}
  Let \(M\)~be a finitely generated submodule of a free module over a multivariate polynomial ring with coefficients in a discrete coherent ring.
  We prove that its module~\(\MLT(M)\) of leading terms is countably generated and give an algorithm for computing explicitly a generating set. This result is also useful when \(\MLT(M)\) is not finitely generated.
  Suppose that the base ring is strongly discrete coherent. We provide a Buchberger-like algorithm and prove that it converges if and only if the module of leading terms is finitely generated. We also state a constructive version of Hilbert's syzygy theorem by following \foreignlanguage{german}{Schreyer}'s method.

  \begin{center}
    \bfseries Résumé
  \end{center}
  
  Soit \(M\) un sous-module de type fini d'un module libre sur un anneau de polynômes multivariés à coefficients dans un anneau cohérent discret.
  Nous prouvons que son module \(\MLT(M)\) des termes dominants est dénombrablement engendré et donnons un algorithme permettant de calculer explicitement un ensemble générateur. Ce résultat est également utile lorsque \(\MLT(M)\) n'est pas de type fini.
Supposons que l'anneau de base soit cohérent fortement discret. Nous proposons un algorithme de type Buchberger et prouvons qu'il converge si et seulement si le module des termes dominants est de type fini. Nous établissons également une version constructive du théorème des syzygies de Hilbert en suivant la méthode de \foreignlanguage{german}{Schreyer}.

\end{abstract}

\bni MSC 2020: 
13D02, 
13P10, 
13C10, 
13P20, 
14Q20. 

\sni Keywords: Generalised \foreignlanguage{german}{Buchberger} algorithm, syzygy theorem, free resolution, monomial order, \foreignlanguage{german}{Schreyer}'s monomial order, \foreignlanguage{german}{Schreyer}'s syzygy algorithm.


\tableofcontents

\section{Introduction}\label{Introduction}

This paper is written in Bishop's style of constructive mathematics (see~\citealt*{Bi67,BB85,MRR,ACMC,Y5}).
It can be seen as a sequel to the paper~\citealt*{GLNY2020} \citep[see also][]{HY,Ye2}. 
While \citet*{AL} require the base ring to be Noetherian, we generalise their results to suitable non-Noetherian settings.

\smallskip Our general setting is the following.

\begin{gencontext}\label{gco}
  In this article, $\R$ is a commutative ring with unit, $X_{1},\dots,X_{n}$ are $n$~indeterminates ($n\geq 1$), $\RuX=\R[X_{1},\dots,X_{n}]$, $\Hm=\RuX^m$ is a free $\RuX$-module with basis~$(e_{1},\dots,e_{m})$ ($m\geq 1$), and $>$~is a monomial order on~$\Hm$ (see Definition~\ref{monomial-module}). The case of an ideal in~$\RuX$ is addressed by considering~$m=1$: $\Hn^1=\RuX$ and $e_1$ is the unit of~$\R$.
\end{gencontext}

We shall need to qualify this context by suitable hypotheses of coherence and discreteness (see Definition~\ref{defdef}).

Let us consider a finitely generated submodule~$M$ of~$\Hm$ and its module~$\MLT(M)$ of leading terms with respect to our monomial order. 

Our first Theorem~\ref{Grob-Buch} states that $\MLT(M)$ is countably generated and yields an algorithm for computing explicitly a generating set. This result is also useful when $\MLT(M)$ is not finitely generated.

The main results of the paper are generalisations of 
Buchberger's algorithm and Schreyer's algorithm. There is a long thread of research in Computer Algebra about generalisations of Buchberger's
algorithm for computing Gröbner bases of ideals and modules and of Schreyer's algorithm for computing the module of syzygies
of submodules of free finitely generated modules over the polynomial ring with coefficients in
a ring. See e.g.\ \citealt{zacharias78,kandri-rodykapur88,pauer07}.

Our second Theorem~\ref{Grob-Buch-algo} indicates a precise context in which the (generalised) Buchberger criterion applies and Buchberger's algorithm computes a Gröbner basis for~\(M\).

Our third Theorem~\ref{schcoherent}
indicates a precise context in which the (generalised) \foreignlanguage{german}{Schreyer} method computes a finite free resolution for~$M$. Hilbert's syzygy theorem is obtained as Corollary~\ref{hilbcoherent}.

These are central problems in Computer Algebra since polynomials or vectors of polynomials always admit 
nontrivial algebraic relations and such algorithms are a basic step, for instance, 
for computing resolutions of modules.

The generalisations are given for so-called
strongly discrete coherent rings, namely rings with membership test for finitely generated
ideals and with finitely generated syzygy modules.

There are many applications of the generalisation of Buchberger's algorithm from the framework of polynomial rings over fields to polynomial rings over more general algebraic structures (Noetherian rings, Euclidean domains and principal ideal domains). The applications go from designing modular
algorithms, formal verification \citep*{brickensteindreyeretal09}, to coding theory and cryptography \citep*{salamoraetal09}, etc. Also many of these
algorithms are currently implemented in Singular and Macaulay2. We hope that our results will be useful in these areas and beyond.

Furthermore, a notable aspect of extending classical theorems in the theory of Gröbner bases to this framework is the emergence of more elementary arguments. Our algorithms terminate because the set of monomials is well ordered and not because of Noetherianity; zerodivisors are dealt with by ``auto-S-polynomials'' (see Remark~\ref{remdefdef2}); ``S-lists'' are introduced as the formal object at work in Schreyer's algorithm (see Definition~\ref{defdef2}).

\section{Constructive definitions and contexts}

We start with recalling the following constructive definitions.

\begin{definition}\label{defdef}
  \leavevmode
  \begin{itemize}

  \item The ring~\(\R\)  is \emph{discrete} if it is equipped with a zero test: equality is decidable.

  \item Let $U$ be an $\R$-module.
    The \emph{syzygy} module of an $n$-tuple $(v_1,\dots,v_n)\in U^n$ is
    \[
      \Syz(v_1,\dots,v_n):= \sotq{(b_1,\dots,b_n) \in \R^{1\times n}}{b_1v_1+\cdots+b_nv_n=0}\text.
    \]
    The syzygy module of a 1-tuple~$v\in U$ is the \emph{annihilator}~$\Ann(v)$ of~$v$.
  \item An $\R$-module~$U$ is \emph{coherent} if the syzygy module of
    every $n$-tuple of elements of~$U$ is finitely generated, i.e.\ if
    there is an algorithm providing a finite system of generators for
    the syzygies, and an algorithm that represents each syzygy as a
    linear combination of the generators. The ring~$\R$ is
    \emph{coherent} if it is coherent as an~$\R$-module. It is
    well-known that a module is coherent if and only if on the one hand any
    intersection of two finitely generated submodules is finitely
    generated, and on the other hand the annihilator of every element is
    a finitely generated ideal.
  \item A ring is \emph{strongly discrete} if it is equipped with a
    membership test for finitely generated ideals, i.e.\ if, given
    $a,b_1,\dots,b_n \in \R$, one can answer the question
    $a\mathrel{\in?}\gen{b_1,\dots,b_n}$ and, in the case of a positive
    answer, one can explicitly provide $c_1,\dots,c_n \in \R$ such that
    $a=b_1c_1+\cdots +b_nc_n$.
  \item $\R$ is a \emph{Bézout ring} if every
    finitely generated ideal is  principal, i.e.\ of the form $\gen a=\R a$ with $a \in \R$. A Bézout ring is strongly discrete if and only if it is equipped with a divisibility test; it is coherent if and only if the annihilator of any element is principal. To be a valuation ring (in the Kaplansky sense) is to be a Bézout local ring \citep[see][Lemma IV-7.1]{ACMC}.

  \item A Bézout ring~$\R$ is \emph{strict} if for all $a_1, a_2 \in\R$ we can find
    $d, a_{1,2} , a_{2,1} , u, v \in\R$ such that $a_1 = d a_{2,1}$, $a_2 = d a_{1,2}$, and $u a_{2,1} + v a_{1,2} = 1$. Valuation rings and strongly discrete Bézout domains are strict Bézout rings; a quotient or a localisation of a strict Bézout ring is again a strict Bézout ring \citep[see][Exercise~IV-7 pp.~220--221, solution pp.~227--228]{ACMC}. A zero-dimensional Bézout ring is strict (because it is a ``Smith ring'', see \citealp*[Exercice~XVI-9 p.~355, solution p.~526]{DLQ2014}, and \citealp[Exercise~IV-8 pp.~221--222, solution p.~228]{ACMC}).
  \end{itemize}
\end{definition}

We shall consider three contexts that are more specific than General context~\ref{gco}.

\begin{discohcontext} \label{discohcontext} General context~\ref{gco} with~$\R$ discrete and coherent.
\end{discohcontext}

It is remarkable that this context suffices for proving Theorem~\ref{Grob-Buch}.

\begin{simpledivicontext} \label{simpledivicontext}
  General context~\ref{gco} with~$\R$ strongly discrete. 
\end{simpledivicontext}

This context is the one commonly considered in the literature.

\begin{remark} \label{remdivicontext}
  We could have considered a ``Division with remainder context'' in which we would assume moreover that we have a partial preorder~$\leq_\R$ on elements of~$\R$
  (with $a\leq_\R b$ if~$a=b$) and a generalised division algorithm~\textsf{Rem} for~$\R$ which computes, for given~$c,c_1,\dots,c_k\in\R$, a remainder $r_0=c-a_1 c_1-\dots-a_kc_k$ satisfying~$r_0=0$ if and only if $c\in\gen{c_1,\dots,c_k}$, and $r_0,a_1c_1,\dots,a_kc_k\leq_\R c$ otherwise. \Sdic~\ref{simpledivicontext} might be seen as the particular case where $\leq_\R$~is equality and \textsf{Rem} returns~\(c\) if $c\notin\gen{c_1,\dots,c_k}$. 
\end{remark}

\begin{sdccontext} \label{sdccontext} General context~\ref{gco} with~$\R$ strongly discrete and coherent.
\end{sdccontext}

This is the convenient context for making Buchberger's and Schreyer's algorithm work.


\section{Gröbner bases for modules over a discrete ring}\label{s1}

\begin{definition}[Monomial orders on finite-rank free {$\RuX$}-modules, \protect{\citealp*[see][]{AL,coxlittleoshea05,Ye8}}, General context~\ref{gco}]\label{monomial-module}
  \leavevmode%
  \begin{enumerate}
  \item\label{monomial-module-1} Monomials, terms.
    \begin{itemize}
    \item   A \emph{monomial} in~$\Hm$ is a vector of the form~$M=\uX^\alpha e_i$ ($1 \leq i \leq m$), where $\uX^\alpha=X_1^{\alpha_1}\cdots X_n^{\alpha_n}$ is a monomial in~$\RuX$;
      the index~$i$ is the \emph{position} of the monomial. The set of monomials in~$\Hm$ is denoted by~$\mathbb{M}_n^m$, with $\mathbb{M}_n^1 \cong \mathbb{M}_n$ (the set of monomials in~$\RuX$). For example, $X_1X_2^3e_2$ is a monomial in~$\Hm$, but~$2X_1e_3$, $(X_1+X_2^3)e_2$, and~$X_1e_2+X_1e_3$ are not.

    \item   If $M=\uX^\alpha e_i$ and $N=\uX^\beta e_j$, we say that $M$ \emph{divides}~$N$ if~$i=j$ and $\uX^\alpha$ divides~$\uX^\beta$. For example, $X_1e_1$ divides~$X_1X_2e_1$, but does not divide~$X_1X_2e_2$. Note that in the case that $M$ divides~$N$, there exists a monomial~$\uX^\gamma$ in~$\mathbb{M}_n$ such that~$N=\uX^\gamma M$: in this case we define
      \(
      N/M\coloneqq \uX^\gamma
      \);
      for example, $(X_1X_2e_1)/(X_1e_1)=X_2$.
    \item A \emph{term} in~$\Hm$ is a vector of the form~$c M$, where $c \in \R \setminus \so{ 0 }$ and $M \in \mathbb{M}_n^m$. We say that a term~$c M$ \emph{divides} a term~$c' M'$, with~$c,c' \in \R \setminus \so{0}$ and~$M,M' \in \mathbb{M}_n^m$, if $c$ divides~$c'$ and $M$ divides~$M'$.

    \end{itemize}

  \item A \emph{monomial order} on~$\Hm$ is a
    relation~$>$ on~$\mathbb{M}_n^m$ such that
    \begin{itemize}
    \item $>$ is  a total order on~$\mathbb{M}_n^m$;
    \item $ \uX^\alpha M > M$ for all~$M \in
      \mathbb{M}_n^m$ and $ \uX^\alpha \in  \mathbb{M}_n \setminus \so{1}$;
    \item $M>  N \implies \uX^\alpha M > \uX^\alpha N $ for all $M,N \in
      \mathbb{M}_n^m$ and $ \uX^\alpha \in  \mathbb{M}_n$.
    \end{itemize}
    Note that, when specialised to the case $m=1$, this definition coincides with the definition of a monomial order on $\RuX$.

  \item Let the ring $\R$ be discrete. Any nonzero vector $u \in \Hm$ can be written as a sum of terms
      \[
        u=c_tM_t+c_{t-1}M_{t-1}+\dots+c_1M_1
      \]
      with $c_1,\dots,c_t \in \R \setminus \so{0}$, $M_1,\dots,M_t \in \mathbb{M}_n^m$, and $M_t > M_{t-1} > \cdots > M_1$.
    \begin{itemize}
    \item We define the \emph{leading coefficient}, \emph{leading monomial}, and \emph{leading term} of $u$ as in the ring case: $\LC (u)=c_t$, $\LM (u)=M_t$, $\LT (u)=\LC (u)\LM (u)$.
    \item   Letting $M_t=\uX^\alpha e_\ell$ with $\uX^\alpha\in \mathbb{M}_n$ and $1\leq \ell \leq m$, we say that $\mdeg(u):=\alpha$ is the \emph{multidegree of~$u$}, that \(\tdeg(u):=\alpha_1+\dots+\alpha_n\) is its \emph{total degree}, and that the index~$\LP(u):=\ell$ is its \emph{leading position}. 

    \item  We stipulate that $\LC (0)=0$, $\LM (0)=0$, and $\mdeg(0)=-\infty$, but we do not define $\LP(0)$.

    \end{itemize}

  \item A monomial order on
    $\RuX$ gives rise to the two
    following canonical
    monomial orders on~$\Hm$.
    Let us consider monomials $M=\uX^\alpha e_i$ and $N=\uX^\beta e_j \in
    \mathbb{M}_n^m$. 
    \begin{itemize}
    \item 
      We say that
      \[
        M>_{\mathrm{TOP}}N \quad \text{if}\quad\left|\,
          \begin{aligned}
            \text{either }&\uX^\alpha > \uX^\beta \\
            \text{or both }&\uX^\alpha = \uX^\beta\text{ and $i<j$.}
          \end{aligned}
        \right.
      \]
      This monomial order is called \emph{term over position} (TOP) because it
      gives precedence to the monomial order on $\RuX$
      over the monomial position. For example, when  $X_2>X_1$,
      we have
      \[
        X_2e_1>_{\mathrm{TOP}} X_2e_2 >_{\mathrm{TOP}} X_1e_1 >_{\mathrm{TOP}} X_1e_2\text.
      \]

    \item  
      We say that
      \[
        M>_{\mathrm{POT}}N \quad \text{if}\quad\left|\,
          \begin{aligned}
            \text{either }&i<j\\
            \text{or both }&i=j\text{ and $\uX^\alpha > \uX^\beta$.}
          \end{aligned}
        \right.
      \]
      This monomial order is called \emph{position over term} (POT) because it
      gives precedence to the monomial position over the
      monomial order on $\RuX$. For example, when
      $X_2>X_1$, we have
      \[
        X_2e_1 >_{\mathrm{POT}} X_1e_1 >_{\mathrm{POT}}X_2e_2 >_{\mathrm{POT}} X_1e_2\text.
      \]
    \end{itemize}
  \end{enumerate}
\end{definition}

\begin{definition}[list and module of leading terms, Gröbner bases]\label{defiGrob0}
  Let $\R$ be a discrete ring and consider a list $G=g_{1},\dots,g_{p}$ in $\Hm$. We denote by \(\LT(G)=\LT (g_{1}),\dots,\LT(g_{p})\) the list of its leading terms. Suppose now that the \(g_i\)'s are nonzero and consider the finitely generated submodule $U= \gen{G}=\R[\uX] g_{1}+\cdots+\R[\uX] g_{p}$  of $\Hm$.
  \begin{enumerate}
  \item The \emph{module of leading terms} of~$U$ is $\MLT (U)\coloneqq\gentq{\LT(u)}{u \in U}$.

  \item $G$ is a \emph{Gröbner basis} for $U$ if $\MLT(U)= \gen{\LT (G)}$.
  \end{enumerate}
\end{definition}

\section{Syzygies over a strict Bézout ring}
\label{sec:syzygy-modules-over}

The following proposition comes from \citealt*{GLNY2020}. We give it as a motivating example showing that the ``obvious'' syzygies do not suffice to generate all of them.
\begin{proposition}\label{lemmsyz0}
  Let \(\R\) be a  strict Bézout ring,  and \(a_1,\dots,a_{s} \in \R\setminus\{0\}\). For $i<j$, write \(a_i=da_{j,i}\), \(a_j=da_{i,j}\),
  and a Bézout identity with respect to~\(a_{i,j}\) and~\(a_{j,i}\). Denote by $(\epsilon_1,\dots,\epsilon_s)$ the canonical basis
  of \(\R^s\). Then
  $\Syz(a_1,\dots,a_{s})$ is generated by the $\tbinom s2$ vectors $a_{i,j}\epsilon_i-a_{j,i}\epsilon_j$ with $i <j$,\footnote{These are the \emph{obvious} syzygies.} together with all the~$z \epsilon_i$
  with~$z \in \Ann(a_i)$. In particular,
  recall that \(\R\) is coherent if and only if $\Ann (a)$ is principal for any
  $a \in \R$: in that case,
  letting $\Ann(a_k)=\gen{b_k} $ for $1 \leq k \leq s$, we have:
  \[\Syz(a_1,\dots,a_{s}) = \gentq{a_{i,j}\epsilon_i-a_{j,i}\epsilon_j,b_{k}\epsilon_k}{1 \leq i < j \leq s, 1\le k\le s}\text.\]

\end{proposition}

\begin{proof}
  Let \(c_1\epsilon_1+\dots+c_{s}\epsilon_s \in \Syz(a_1,\dots,a_{s})\), and let $\rs(a_i,a_j):=a_{i,j}\epsilon_i-a_{j,i}\epsilon_j$. 
  For each permutation $i_1,\dots,i_{s}$ of $1,\dots,s$, consider the product~\(a_{i_1,i_2}\*\cdots a_{i_{s-1},i_{s}}\). We claim that there is a Bézout identity for these products, so that it suffices to rewrite the expression $a_{i_1,i_2}\*\cdots a_{i_{s-1},i_{s}}\*(c_1\epsilon_1+\cdots+c_{s}\epsilon_{s})$ in terms of \(\rs(a_{i_1},a_{i_{2}})\), \dots, \(\rs(a_{i_{s-1}},a_{i_s})\), and \(\Ann(a_{i_s})\epsilon_{i_s}\): let us
  replace successively
  \[
    \begin{aligned}
      a_{i_1,i_2}&\epsilon_{i_1}&&\text{by}&&\rs(a_{i_1},a_{i_2})+a_{i_2,i_1}\epsilon_{i_2}\text,\\
      \vdots&&&&&\hphantom{\rs(a_{i_1},a_{i_2})}\vdots\\
      a_{i_{s-1},i_{s}}&\epsilon_{i_{s-1}}&&\text{by}&&\rs(a_{i_{s-1}},a_{i_{s}})+a_{i_{s},i_{s-1}}\epsilon_{i_{s}}\text.a_i,a_j
    \end{aligned}
  \]
  At the end, the sum will be a linear combination of $\rs(a_{i_1},a_{i_2})$, \dots, $\rs(a_{i_{s-1}},a_{i_{s}})$, and~$\epsilon_{i_{s}}$; let $z$ be the coefficient of $\epsilon_{i_{s}}$ in this combination. As \( c_1\epsilon_1+\dots+c_{s}\epsilon_s \in \Syz(a_1,\dots,a_{s})\), we have $z\epsilon_{i_{s}} \in \Syz(a_1,\dots,a_{s})$, i.e.\ $za_{i_{s}}=0$.

  It remains to obtain the Bézout identity for the products $a_{i_1,i_2}\cdots a_{i_{s-1},i_{s}}$. For this, it is enough to develop the product of the
  $\tbinom s2$~Bézout identities with respect to~$a_{i,j}$ and~$a_{j,i}$,
  $1\leq i<j\leq s$: this yields a sum of products of
  $\tbinom s2$~terms, each of which is either~$a_{i,j}$ or~$a_{j,i}$,
  $1\leq i<j\leq s$, so that it is indexed by the tournaments on the
  vertices~$1,\dots,s$ (i.e.\ the directed graphs with exactly one edge between each two vertices); every such product contains a product of the above
  form~$a_{i_1,i_2}\cdots a_{i_{s-1},i_{s}}$ because every tournament
  contains a Hamiltonian path \citep[see][]{redei35}. \end{proof}

\begin{remark}
  The above proof results from an analysis of the following proof in the case where $\R$ is local, which entails in fact the general case.
  Since $\R$ is a
  valuation ring, we may consider a permutation $i_1,\dots,i_{s}$ of $1,\dots,s$ such that $a_{i_s} \divides a_{i_{s-1}} \divides
  \cdots \divides a_{i_1}$. Thus $\rs(a_{i_1},a_{i_2})= \epsilon_{i_1} - a_{i_2,i_1}\epsilon_{i_2}$, \dots, $\rs(a_{i_{s-1}},a_{i_s})=\epsilon_{i_{s-1}}-a_{i_s,i_{s-1}}\epsilon_{i_s}$ for
  some~$a_{i_2,i_1},\dots,a_{i_s,i_{s-1}}$. Then, by replacing successively $\epsilon_{i_k}$ by $\rs(a_{i_k},a_{i_{k+1}})+a_{i_{k+1},i_k}\epsilon_{i_{k+1}}$, every syzygy $c_1\epsilon_1+\dots+c_{s}\epsilon_s$ may be rewritten as a linear combination of $\rs(a_{i_1},a_{i_2})$, \dots, $\rs(a_{i_{s-1}},a_{i_s})$, and $\epsilon_{i_s}$, with the coefficient of~$\epsilon_{i_s}$ turning out to lie in $\Ann(a_{i_s})$.
  \eoe
\end{remark}

In the following, we give examples of coherent rings over which syzygy modules are not always generated by
vectors with at most $2$ nonzero components.

\begin{example}\label{2-comp-1}
  Consider the (Noetherian) coherent ring $\mathbb{Z}[u]$ and
  the syzygy module $\Syz( 2,u,u+2)$. As $\Syz( 2,u)=\gen{(u,-2)}$,
  $\Syz( 2,u+2)=\gen{(u+2,-2)}$, and $\Syz( u+2,u)=\gen{(-u,u+2)}$,
  we conclude that if $s=(s_1,s_2,s_3) \in \Syz( 2,u,u+2)$ can be
  written as a $\mathbb{Z}[u]$-linear combination of syzygies in
  $\Syz( 2,u,u+2)$ with at most $2$ nonzero components, then it has
  entries $s_i$ in $\gen{2,u}$. The syzygy
  $(1,1,-1) \in \Syz(2,u,u+2)$ does not satisfy this property.
\end{example}
\begin{example}
  Consider the ring $\R= \mathbb{Z}[u] + v \, \mathbb{Q}(u)[v]_{(v)}$. 
  It is coherent by \citet[Theorem~3, since
  {${\rm q.f.}(\mathbb{Z}[u])=\mathbb{Q}(u)$} and {$\mathbb{Z}[u]$} is coherent]{dobbspapick76} but
  non-Noetherian \citep[since {$\mathbb{Z}[u]$} is not a field, see][§~17, Exercise~14]{gilmer72}. As in Example~\ref{2-comp-1}, $(1,1,-1) \in \Syz_{\R}(2,u,u+2)$ cannot be
  written as an $\R$-linear combination of syzygies in~$\Syz_{\R}( 2,u,u+2)$ with
  at most $2$ nonzero components (suppose so and  take $v=0$).
\end{example}


\section{Syzygies in a polynomial ring over a discrete coherent ring}\label{seccrucial}

\subsection{Syzygies of terms}
\label{sec:syzygies-terms}

\begin{definition}[syzygies of terms, Discrete coherent context~\ref{discohcontext}]
  \label{defisyz-terms0}
  Let $p\geq1$ and $\Pp=\sotq{E}{\emptyset\neq E\subseteq \lrb{1\twodots p}}$ be the set of nonempty subsets of the set of indices $\lrb{1\twodots p}=\{1,\dots,p\}$.
  Consider $M_1=M'_1e_{i_1}$, \dots, $M_p=M'_pe_{i_p}$ monomials
  in~$\Hm$, $a_1,\dots,a_p \in \R$.  Let
  $\Pap(M_1,\dots,M_p)\subseteq\Pp$ be the subset of \emph{position level sets} of $(M_1,\dots,M_p)$, i.e.\ of the nonempty subsets~\(E\) of~$\lrb{1\twodots p}$ such that
  $i_{j}=i_{j'}$ for $j,j'\in E$. Note that all singletons belong to $\Pap(M_1,\dots,M_p)$. For each position level set~$E$ of $(M_1,\dots,M_p)$, let
  $s^E_{1},\dots,s^E_{\ell^E}$ be a finite number of generators
  of~$\Syz((a_j)_{j \in E})$ as given by a certificate of coherence
  for~$\R$; here $s^E_{i}=(s^E_{i,j})_{j\in E}$. Let
  $M^E=\lcm(M_j\mathrel;j \in E)$ and
  \({S}^E(a_1M_1,\dots,a_pM_p)\) be the list \(S^E_{1},\dots,S^E_{\ell^E}\),
  where $S^E_{i}=(S^E_{i,1},\dots,S^E_{i,p})$ with
  \[
    S^E_{i,j}=\begin{cases}
      {s}^E_{i,j}\ M^E/M_j&\text{if $j \in E$,}\\
      0&\text{otherwise.}
    \end{cases}
  \]
  This is a syzygy for $(a_1M_1,\dots,a_pM_p)$: see
  Equation~\eqref{SEisyzygy} below. Finally, let
  $S(a_1M_1,\dots,a_pM_p)$ be the concatenation of all the lists
  $S^E(a_1M_1,\dots,a_pM_p)$ when $E$ ranges over the position level sets~$E$ of $(M_1,\dots,M_p)$.
\end{definition}

\begin{example}\label{exZ8Z} Let $T_1=(2X^2Y,\,0)=2M_1$, $T_2=(XY^2,\,0)=M_2$, $T_3=(0,\,4X)=4M_3$ in $(\mathbb{Z}/8\mathbb{Z})[X,Y]^2$. We have
  $\Pap(M_1,M_2,M_3)=\left\{ \{1\},\, \{2\},\,\{3\},\,\{1,2\} \right\}$, $S_1^{\{1\}}=(4,0,0)$, $S_1^{\{2\}}=(0,0,0)$, $S_1^{\{3\}}=(0,0,2)$,
  $S_1^{\{1,2\}}=(Y,6X,0)$, and $\ell^{\{1\}}=\ell^{\{2\}}=\ell^{\{3\}}=\ell^{\{1,2\}}=1$.
\end{example}

\begin{definition}[leading monomial of summands and their leading monomial index set]\label{defisyz-terms1}
  Let $f_1,\dots,f_p \in \Hm$ not all zero with $\LC(f_j)=a_j$
  and $\LM(f_j)=M_j$.  Let $g_1,\dots,g_p\in\RuX$ with
  $\LC (g_j)=b_j$ and $\LM (g_j)=N_j$.  The \emph{leading monomial of
    the summands of~$g_1f_1+\dots+g_pf_p$}
  is the monomial $\aLM=\aLM(\ug)=\sup_{j\in\lrb{1\twodots p}} N_jM_j$,\footnote{Note that this is not in general the leading monomial of the vector that is the sum of the expression.} and their \emph{leading monomial index set} is $E=\sotq{j} {N_jM_j=\aLM}$.
\end{definition}
  
Propositions~\ref{syz-terms} and~\ref{syz-terms-2} as well as Theorem~\ref{Grob-Buch} generalise the method of \citealt*[Theorem 4.5]{GLNY2020}.

\begin{proposition}[Discrete \coh context~\ref{discohcontext}]\label{syz-terms}
  The finite
  list \(S(a_1M_1,\dots,a_pM_p)\)
  generates the syzygy module \(\Syz(a_1M_1,\dots,a_pM_p)\subseteq \RuX^p\).
  \end{proposition}

\begin{proof}
  Let us use the notation of Definitions~\ref{defisyz-terms0} and~\ref{defisyz-terms1}. We first check that each $S^E_{i}$ is a syzygy for \((a_1M_1,\dots,a_pM_p)\):
  \begin{equation}
    S^E_{i,1}a_1M_1+\dots+S^E_{i,p}a_pM_p=\som_{j\in E}s^E_{i,j}a_jM^E=\Bigl(\som_{j\in E}s^E_{i,j}a_j\Bigr)M^E=0\text.\label{SEisyzygy}
  \end{equation}
  Conversely, let $\ug=(g_1,\dots,g_p) \in \Syz(a_1M_1,\dots,a_pM_p)$,  not all $g_j$ zero, let $\aLM=\aLM(\ug)$ be the leading monomial of the summands of $g_1a_1M_1+\dots+g_pa_pM_p$,
  and let $E$ be their leading monomial index set.
  We have
  $\som_{j\in E}b_j a_j =0$, and thus $(b_j)_{j \in E}=c_1 s^E_{1}+ \cdots+ c_{\ell^E} s^E_{\ell^E}$ for some $c_1,\dots,c_{\ell^E} \in \R$. Let
  \[\ugprime=(g'_1,\dots,g'_p)=\ug- \frac{\aLM}{M^E} \som_{i=1}^{\ell^E} c_i S^E_{i} \in \Syz(a_1M_1,\dots,a_pM_p)\text;\]
  note that $M^E$ divides $\aLM$ because every $M_j$, $j\in E$, does. We have $g'_j=g_j$ for $j \notin E$ and, for
  $j\in E$,
  \[
    \begin{multlined}
      g'_j= g_j-\frac{\aLM}{M^E} \sum_{i=1}^{\ell^E} c_i S^E_{i,j}= g_j-\frac{\aLM}{M^E} \sum_{i=1}^{\ell^E} c_i
      \frac{M^E}{M_j}{s}^E_{i,j}\\
      = g_j-\frac{\aLM}{M_j} \sum_{i=1}^{\ell^E} c_i {s}^E_{i,j}=
      g_j-\frac{\aLM}{M_j} b_j= g_j- \LT(g_j)\text.
    \end{multlined}
  \]

  \n Thus $\aLM(\ugprime)< \aLM(\ug)$. Reiterating this (with $\ugprime$ instead of $\ug$), we reach the desired result after a finite number of steps since the set of monomials is
  well ordered.
\end{proof}

\begin{example}[Example~\ref{exZ8Z} continued]\label{exZ8Zbis} In $(\mathbb{Z}/8\mathbb{Z})[X,Y]^2$, we have
  \[ \Syz(T_1, T_2, T_3)=\gen{(4,0,0),(0,0,2),(Y,6X,0)}\text.\]
\end{example}

Following in detail the first step in the preceding proof we get the following proposition: given $f_1,\dots,f_p \in \Hm$ not all zero and $g_1,\dots,g_p\in\RuX$ such that the leading monomial of the summands \(g_jf_j\) is bigger than the leading monomial
of their sum, we can construct a new expression for this sum involving more elements of \(\gen{f_1,\dots,f_p}\) such that the leading monomial of its summands becomes smaller.

\begin{proposition}[notation of Definitions~\ref{defisyz-terms0} and~\ref{defisyz-terms1}, Discrete \coh context~\ref{discohcontext}]\label{syz-terms-2}
  Let
  \[
    u=\som_{j \in\lrb{1\twodots p}}g_jf_j\text,\enskip\aLM=\sup\nolimits_{j \in\lrb{1\twodots p}}N_jM_j\text,\enskip\text{and}\enskip E=\sotq{j\in\lrb{1\twodots p}}{N_jM_j=\aLM}\text.
  \]
  If  $\LM(u) <\aLM$, then \(\som_{j\in E}b_j a_j =0\); let $c_1,\dots,c_{\ell^E} \in \R$ such that $(b_j)_{j \in E}=c_1 s^E_{1}+ \cdots+ c_{\ell^E} s^E_{\ell^E}$. Then $f_{p+1},\dots,f_{p+\ell^E}  \in \Hm$ and $g_{p+1},\dots,g_{p+\ell^E}\in \RuX$ defined by
  \[f_{p+i}=  \som_{j \in E}   {S}^E_{i,j}\, f_j\text{ and }g_{p+i}=c_i \frac{\aLM}{M^E}\]
  are such that
\[u=\som_{j \in E}(g_j-\LT(g_j))f_j+\som_{j \in\lrb{1\twodots p+\ell^E}\setminus E}g_jf_j\]
is an expression for~$u$ whose summands have leading monomial $< \aLM$.
\end{proposition}

\begin{proof}
  As $\LM(u) <\aLM$, the coefficient of~$\aLM$ in $\som_{j\in E}g_jf_j$ vanishes, so that $\som_{ j \in E}b_ja_j=0$: we have
  \[
    \begin{aligned}
      \som_{j \in E}\LT(g_j)f_j
      &= \som_{j \in E}\som_{1\leq i \leq \ell^E} c_i {s}^E_{i,j} N_jf_j\\
      &= \som_{1\leq i \leq \ell^E} c_i\som_{j \in E} {S}^E_{i,j} \frac{M_jN_j}{M^E}\,f_j
      =\som_{1\leq i \leq \ell^E}g_{p+i} f_{p+i}
    \end{aligned}
  \]
  with
  \[
    \LM(g_{p+i})\LM(f_{p+i}) \leq \frac {\aLM}{M^E} \LM\Big(\som_{j \in E}   {S}^E_{i,j}\, f_j\Big) < \frac{\aLM}{M^E}M^E=\aLM\text.\qedhere
  \]
\end{proof}

\subsection{Examples in the case of an ideal}

This case is Discrete \coh context~\ref{discohcontext} with $m=1$; every subset of~$\Pp$ is a position level set.

\begin{example}Let us consider the following syzygy of~\((6XY^2,15X^2YZ,10Z^2)\) in $\mathbb{Z}[X,Y,Z]$:
  \[
    \ug=(g_1,g_2,g_3)=(5XZ+10Z^2, -2Y+2Z, -3X^2Y-6XY^2)\text.
  \]
  Following the algorithm given in the
  proof of Proposition~\ref{syz-terms} and considering the graded monomial
  lexicographic order with $X>Y>Z$, we have
  $\aLM(\ug)=\aLM=X^2Y^2Z$, $E=\{1,2\}$,
  $\Syz(6,15)=\gen{s^E_{1}=\frac{1}{3}(-15,6)=(-5,2)}$,
  $\ell^E=1$, $M^E=X^2Y^2Z$,
  $S^E_{1}=\bigl(-5 \frac{X^2Y^2Z}{XY^2},\, 2
  \frac{X^2Y^2Z}{X^2YZ},\,0\bigr)=(-5XZ,\,2Y,\,0)$,
  $(b_1,b_2)=(5,-2)=(-1)\cdot s^E_{1}$, \(c_1=-1\),
  $\ugprime=(g'_1,g'_2,g'_3)=\ug- \frac{\aLM}{M^E}
  \sum_{i=1}^{\ell^E} c_i S^E_{i}={\ug}+\frac{X^2Y^2Z}{X^2Y^2Z} S^E_{1}= \ug+(-5XZ,\,
  2Y,\, 0)=(10Z^2,\, 2Z,\, -3X^2Y-6XY^2)=(g_1- \LT (g_1),\,g_2- \LT
  (g_2),\, g_3)$, with
  $\aLM(\ugprime)=\aLM'=X^2YZ^2 < \aLM(\ug)$.

  \smallskip \n Continuing with $\ugprime$, we obtain
  $E'=\{2,3\}$,
  $\Syz(15,10)=\gen{s^{E'}_{1}}$ with $s^{E'}_{1}=\frac{1}{5}(-10,15)=\allowbreak(-2,3)$,
  $\ell^{E'}=1$, $M^{E'}=X^2YZ^2$,
  $S^{E'}_{1}=\bigl(0, -2 \frac{X^2YZ^2}{X^2YZ},\, 3
  \frac{X^2YZ^2}{Z^2}\bigr)=(0,\,-2Z,\,3X^2Y)$,
  $(b_2',b_3')=(2,-3)=(-1)\cdot s^{E'}_{1}$,
  $\ugseconde=\ugprime- \frac{\aLM'}{M^{E'}}
  \sum_{i=1}^{\ell^{E'}} c_i' S^{E'}_{i}=\ugprime+\frac{X^2YZ^2}{X^2YZ^2} S^{E'}_{1}= \ugprime+(0,\,
  -2Z,\, 3X^2Y)=(10Z^2,\, 0,\, -6XY^2)=(g'_1, g'_2- \LT (g'_2),\,g'_3-
  \LT (g'_3))=-2 S^{\{1,3\}}_{1}$. We conclude that
  \[ \ug=-S^{\{1,2\}}_{1}-S^{\{2,3\}}_{1}-2 S^{\{1,3\}}_{1}\text.\]
\end{example}

\begin{example}
  Let us consider the following syzygy of~\((3XY,3Y,X)\) in~\(\mathbb{Z}[X,Y]\):
  \[
    \ug=(g_1,g_2,g_3)=(2X+Y,\, -3X^2+2XY,\, 3XY-9Y^2)\text.
  \]
  Following the
  algorithm given in the proof of Proposition~\ref{syz-terms} and
  considering the lexicographic monomial order with $X>Y$, we have
  $\aLM(\ug)=\aLM=X^2Y$, $E=\{1,2,3\}$,
  $\Syz(3,3,1)=\gen{s^E_{1}=(-1,1,0),\, s^E_{2}=(-1,0,3)}$,
  $\ell^E=2$, $M^E=XY$, $S^E_{1}=(-1,X,0),\, S^E_{2}=(-1,0,3Y)$,
  $(b_1,b_2,b_3)=(2,-3,3)=-3s^E_{1} + s^E_{2}$, $(c_1,c_2)=(-3,1)$,
  $\ugprime=(g'_1,g'_2,g'_3)=\ug- \frac{\aLM}{M^E}
  \sum_{i=1}^{\ell^E} c_i S^E_{i}=\ug-\frac{X^2Y}{XY}
  (-3S^E_{1}+ S^E_{2})=\ug-X(2,-3X,3Y)=(Y,\, 2XY,\,
  -9Y^2)=({g_1- \LT (g_1)},{g_2- \LT (g_2)}, {g_3- \LT (g_3)})= 2Y
  S^E_{1}-3Y S^E_{2} $, with
  $\aLM(\ugprime)=XY^2 < \aLM(\ug)$.  We conclude that
  \[ \ug=(-3X+2Y)S^{\{1,2,3\}}_{1}+(X-3Y) S^{\{1,2,3\}}_{2}\text.\]
\end{example}

\subsection{S-lists and iterated S-lists, a fundamental theorem}

\begin{definition}[Discrete \coh context~\ref{discohcontext}]\label{defdef2}
  Let \(f_1,\dots,f_p\in\Hm\setminus\{0\}\) and consider their leading terms \(a_1M_1,\dots,a_pM_p\in\Hm\).
  If $S_1,\dots, S_{\ell}$ is the list of generators of $\Syz(\LT(f_1,\dots,f_p))$ computed in Proposition~\ref{syz-terms}, the  \emph{S-list} of $f_1,\dots,f_p$ is the list
  \[\St(f_1,\dots,f_p)=S_{1,1}f_1+\dots+S_{1,p}f_p,\dots,S_{\ell,1}f_1+\dots+S_{\ell,p}f_p\]
  after having deleted the vanishing items.
  By induction, we define the \emph{iterated S-lists} by
  \[
    \left\{
      \begin{aligned}\St^0(f_1,\dots,f_p)&=f_1,\dots,f_p\text;\\
        \St^{q+1}(f_1,\dots,f_p)&=\text{the concatenation of $\St^{q}(f_1,\dots,f_p)$ with $\St(\St^{q}(f_1,\dots,f_p))$.}
      \end{aligned}
    \right.
  \]
Note that each item of an iterated S-list is in~$\gen{f_1,\dots,f_p}$.
\end{definition}

\begin{remark}\label{remdefdef2}  If $\R$ is a Bézout ring then for any $a_1,\dots,a_q \in \R$ there exists a finite generating set for $\Syz( a_1,\dots,a_q)$  whose
  vectors have at most two nonzero components (see Proposition~\ref{lemmsyz0}). Choose such a generating set of syzygies in $\R^q$ at each iteration. Then there are only S-pairs ($\# E=2$) and auto-S-polynomials ($\# E=1$) in the corresponding iterated S-lists of $f_1,\dots,f_p$, as expected.
  Similarly, if $\R$ is a Prüfer domain 
  (e.g.\ $\R=\{f\in\QQ[X]\mathrel;f(\ZZ)\subseteq\ZZ\}$, which has Krull dimension equal to~$2$, see \citealt{zbMATH06938583,Duc2O15}),  
  then for any $a_1,\dots,a_q \in \R$ there exists a finite generating set for $\Syz( a_1,\dots,a_q)$  whose
  vectors have at most two nonzero components: in fact, a Prüfer domain is  locally a valuation domain (thus, locally a Bézout domain). Then there are only S-pairs in the corresponding iterated S-lists of $f_1,\dots,f_p$ (the auto-S-polynomials vanish since the ring $\R$ is supposed to be integral).
  \eoe
\end{remark}

\begin{maintheorem}[Discrete \coh context~\ref{discohcontext}]
  \label{Grob-Buch}
  Let \(f_1,\dots,f_p\in\Hm\setminus\{0\}\). For any \(u\in\gen{f_1,\dots,f_p}\) there exist
  $q \in \mathbb{N}$ and items~$p_1,\dots,p_{t}$ in the list $\St^{q}(f_1,\dots,f_p)$
  such that $ \LT(u)\in \gen{\LT(p_1,\dots,p_{t})}$.  In other
  words,
  \[ \MLT(\gen{f_1,\dots,f_p})= \mathop{\bigcup\big\uparrow}_{q \in \mathbb{N}}\, \gen{\LT(\St^q(f_1,\dots,f_p))}\text.\]
\end{maintheorem}

\begin{comment} \label{comGrob-Buch}
  Compared to Theorem 4.2.8 of \citet{AL}, who suppose that the base ring $\R$ is strongly discrete, coherent, and  Noetherian, 
  our Theorem~\ref{Grob-Buch} supposes only that $\R$ is discrete and coherent. Moreover, we do not perform divisions. This could be useful when one tries 
  to prove results on the structure of the ideals of leading terms (see \citealp{BY1,GuY,Ye6,Ye7} and the recent solution of the Gröbner ring conjecture in \citealp{yengui24}). However  Theorem~\ref{Grob-Buch} does not give a termination condition when 
  one knows that the leading terms ideal is finitely generated (such a condition is given in Theorem~\ref{Grob-Buch-algo}). Our Theorem~\ref{Grob-Buch} is also useful when the leading terms ideal is not
  finitely generated (see Example~\ref{ex}~\eqref{ex1} and the counterexample given in \citealp{Ye6}).
  \eoe
\end{comment}

\begin{proof}
  Write
  \begin{equation}\label{eqlospol}
    u=\som_{j=1}^p g_j  f_j  \hbox{ with }  N_j=\LM (g_j)\hbox{  and } M_j=\LM (f_j)\text.
  \end{equation}
  So $\LM (u)\leq \sup_{1\leq j \leq p}(N_jM_j)=:\aLM$ (the
  leading monomial of the summands of~$u$ in~\eqref{eqlospol}).

  \n Case 1. $\LM (u)= \aLM$.  Clearly  $\LT (u) \in   \gen{\LT(f_{1},\dots,f_{p})}$.

  \noindent  Case 2.  $\LM (u) <\aLM$. Let $E=\{\,j\mathrel;N_jM_j=\aLM\,\}$. By virtue of Proposition~\ref{syz-terms-2},  we obtain another expression for $u$,
  \begin{equation}
    \label{eq:1}
    u=\som_{j \in E}(g_j-\LT(g_j))f_j+\som_{j \in\lrb{1\twodots p+\ell^E}\setminus E}g_jf_j
  \end{equation}
  with the $f_j$'s in $\St^1(f_1,\dots,f_p)$ and the $g_j$'s in $\RuX$, and the leading monomial of the summands of~$u$ in~\eqref{eq:1} is $ < \aLM$. Reiterating this, we end up with a situation like that of Case 1  because the set of monomials is
  well ordered. So we reach the desired result after a finite number of steps.
  \end{proof}

\begin{remark}
  In the proof of Theorem~\ref{Grob-Buch} with
  $m=1$, if the considered monomial order refines total degree
  (i.e.\ if $M > N$ whenever $\tdeg(M) > \tdeg(N)$), then, letting $d=
  \max_{1 \leq j \leq p} \bigl( \tdeg(g_{j})+\tdeg(f_j)\bigr)
  $ and $\delta=\tdeg(u)$ (assumed $\geq 1$), we have $q \leq \binom{n+d}{d}- \binom{ n+\delta-1}{\delta-1}$ (the number of monomials in
  $X_1,\dots,X_n$ of total degree at least $\delta$ and at most $d$).
  \eoe
\end{remark}

\begin{example}
  \label{ex}
  Let $\V$ be a nonarchimedean valuation domain, i.e.\ a valuation
  domain $\V$ such that there exist nonunits $a,b \in \V$ with $a^q$
  dividing $b$ for every $ q \in \mathbb{N}$.
  \begin{enumerate}
  \item\label{ex1} Let $f_1=aX+1,\,f_2=b\in\V[X]$.  Then
    $\MLT (\gen{f_1,f_2})$ is not finitely generated \citep[see][Example
    253]{Y5}: $\LT(\St^q(f_1,f_2))=\geN{aX,b,\frac{b}{a},\dots,\frac{b}{a^q}}$ and
    $\MLT(\gen{f_1,f_2})=\geN{aX,b,\frac{b}{a},\frac{b}{a^2},\dots}$.
  \item Let $f_{1}=a^2+aXY,\,f_{2}=bY^2\in\V[X,Y]$. We have
    \[
      \begin{multlined}
        \geN{\LT (\St^0(f_{1},f_{2})}=aY\geN{X,\frac{b}{a}Y} \subsetneq
        \geN{\LT (\St^1(f_{1},f_{2})}=aY\geN{X,\frac{b}{a}Y,b}\\
        \subsetneq
        \geN{\LT (\St^2(f_{1},f_{2})}
        = \gen{aXY, bY^2,abY,a^2b}= \MLT(\gen{f_{1},f_{2}})\text.
      \end{multlined}
    \]
  \end{enumerate}
\end{example}

\begin{corollary}[Discrete \coh context~\ref{discohcontext}]
  \label{Grob-Buch22}
  Let \(I = \gen{f_1,\dots,f_p}\) be a nonzero
  finitely generated submodule of~$\Hm$. Suppose that $\MLT(I)$ is finitely generated,
  i.e.\ that there exist $u_{1},\dots,u_{t} \in I$ such that
  $\MLT(I)=\gen{\LT(u_{1},\dots,u_{t}) }$. Then there exists
  $q \in \mathbb{N}$ such that $ \MLT(I)= \gen{\LT(\St^q(f_1,\dots,f_p))}$.
\end{corollary}

\begin{remark} In Corollary~\ref{Grob-Buch22} with $m=1$, if the
  considered monomial order refines total degree, then, writing
  $u_{k}=\som_{j=1}^p g_{k,j} f_j$ with $g_{k,j} \in \RuX$ and
  letting
  $d= \max_{1 \leq k\leq t, 1 \leq j \leq p}
  (\tdeg(g_{k,j})+\tdeg(f_j)) $ and
  $\delta=\min_{1 \leq k\leq t} \tdeg(u_{k})$ (assumed to be $\geq 1$), we
  have $q \leq \binom{n+d}{d}- \binom{ n+\delta-1}{\delta-1}$.
  \eoe
\end{remark}

\section{Basic algorithms}\label{sec:basic-algorithms}

\subsection{The division algorithm}\label{sec:division-algorithms}

This algorithm in Simple division context~\ref{simpledivicontext} needs $\R$ to be strongly discrete; note  that coherence is not used here. Like the classical division algorithm for $\F[\uX]^m$ with $\F$ a discrete field \citep[see][Algorithm~211]{Y5}, this algorithm has the following goal.
\[
  \begin{tabular}{@{}rl@{}}
    \lstinline|Input|&$u\in\Hm$, $h_{1},\dots,h_{p} \in \Hm\setminus \so 0$.\\
    \lstinline|Output|&$q_1,\dots,q_p \in \RuX$ and $r \in \Hm$ such that\\
                     &$\left\{\begin{aligned}
                         &u=q_1h_{1}+\cdots+q_ph_{p}+r\text,\\
                         &\LM (u)\geq \LM (q_j) \LM (h_{j})\text{ whenever $q_j\ne0$,}\\
                         &T\notin \gen{\LT(h_{1},\dots,h_{p})} \text{ for each term~$T$ of $r$.}
                       \end{aligned}\right.$
  \end{tabular}
\]

\begin{definition} \label{notaremainder}
  The vector $r$ is called \emph{a remainder of \(u\) on division by the list~\(H= \lst{h_{1},\dots,h_{p}}\)} and is denoted by $r= \overline{u}^H$.
\end{definition}

Algorithm~\ref{gendivalg1} provides a suitable answer: a suitable remainder $r$ and suitable quotients~$q_j$. Nevertheless, there are a priori many different possible answers.

\begin{divalgorithm}[\Sdic~\ref{simpledivicontext}]\label{gendivalg1} 
  \leavevmode
\begin{lstlisting}
Division($u,h_{1},\dots,h_{p}$)
local variables $j:\lrb{1\twodots p}$, $D:\text{subset of }\lrb{1\twodots p}$,
                $c,c_1,\dots,c_p,a_1,\dots,a_p:\R$, $M,M_1,\dots,M_p:\Hm$;
$r\gets0$;
for $j$ from $1$ to $p$ do $q_j\gets0$; $M_j\gets \LM(h_{j})$; $c_j\gets \LC(h_{j})$ od;
while $u\ne0$ do
  $M\gets \LM (u)$; $c\gets \LC(u)$; $D\gets\sotq{j}{M_j \divides M}$;
  if $c\in \gen{c_j;j\in D}$  then
    find $(a_j)_{j\in D}$ such that $\sum_{j\in D}a_jc_j=c$;
    $u\gets u-\som_{j\in D} a_j(M/M_j) h_{j}$;
    for $j\in D$ do $q_j\gets q_j+a_j(M/M_j)$ od
  else $r \gets r +  cM$; $u\gets u-cM$ fi
od;
return $r,q_1,\dots,q_p$
\end{lstlisting}
\end{divalgorithm}

One checks by induction that $\LM(q_j)\LM(h_{j})\leq \LM(u)$ and $u=q_1h_{1}+\cdots+q_ph_{p}+r$.

\subsection{Syzygy algorithms}\label{sec:syzterms-algor}

\begin{notation} \label{notaSeq}
  We denote by $\List(A)$ the set of (finite) lists of elements of $A$.
\end{notation}

These algorithms take place in Discrete \coh context~\ref{discohcontext}.
They are a key tool for constructing a Gröbner
basis and have been introduced by \citet{Buchbe0} for the case
where the base ring is a discrete field.

\smallskip We begin with the basic syzygy algorithm returning $S^E(a_1M_1,\dots,a_pM_p)=S^E_{1},\dots,S^E_{\ell}$ for \(a_1M_1,\dots,a_pM_p\in \Hm\) and a subset $E\subseteq \Pap(M_1,\dots,M_p)$: see Definition~\ref{defisyz-terms0}. Let us recall that  $\Hn^1=\RuX$.
\[\makebox[\textwidth]{%
    \begin{tabular}{@{}rl@{}}
      \lstinline|Input|&$a_1M_1,\dots,a_pM_p$ terms in $ \Hm$, $E\subseteq \Pap(M_1,\dots,M_p)$.\\
      \lstinline|Output|& A list of syzygies \((S_{1,j}^E)^{\vphantom{E}}_{j\in\lrb{1\twodots p}},\dots,(S_{\ell,j}^E)^{\vphantom{E}}_{j\in\lrb{1\twodots p}}\)
                          for $(a_1M_1,\dots,a_pM_p)$\\
                       &such that $S^E_{i,j}=0$ for $j\notin E$ and,
                         for every syzygy expression\\
                       & $g_1a_1M_1+\dots+g_pa_pM_p$ whose summands have leading monomial\\
                       &index set~$E$, $(\LT(g_j))_{j\in E}\in\gen{(S^E_{1,j})^{\vphantom{E}}_{j\in E},\dots,(S^E_{\ell,j})^{\vphantom{E}}_{j\in E}}$.
    \end{tabular}}
\]

\begin{basicsyzalgorithm}[basic syzygies of terms, Definition~\ref{defisyz-terms0}, Discrete coherent context~\ref{discohcontext}]\label{genSyzBasic}
  \leavevmode
\begin{lstlisting}
BasicSyzygiesOfTerms($a_1M_1,\dots,a_pM_p,E$)
local variables $j:\lrb{1\twodots p}$, $\ell,i:\NN$, $s_1,\dots,s_\ell:\R^E$, $M^E:\Hm$;
find $\ell,s_{1},\dots,s_{\ell}$ such that $\Syz((a_j)_{j \in E})=\gen{s_{1},\dots,s_{\ell}}$;
for $i$ from $1$ to $\ell$ do
  for $j$ from $1$ to $p$ do
    $M^E\gets \lcm(M_j\mathrel; j \in E) $;
    if $j\in E$ then $S_{i,j}^E\gets s_{i,j}(M^E/M_j)$ else $S_{i,j}^E\gets 0$ fi
  od
od;
return $(S_{1,j}^E)^{\vphantom{E}}_{j\in\lrb{1\twodots p}},\dots,(S_{\ell,j}^E)^{\vphantom{E}}_{j\in\lrb{1\twodots p}}$
\end{lstlisting}
\end{basicsyzalgorithm}

\smallskip We now give an algorithm whose goal is to provide a generating set of syzygies for a vector of terms in $\Hm$; see Definition~\ref{defisyz-terms0} and Proposition~\ref{syz-terms}.

\[\makebox[\textwidth]{%
    \begin{tabular}{@{}rl@{}}
      \lstinline|Input|&$a_1M_1=a_1M'_1e_{i_1},\dots,a_pM_p=a_pM'_pe_{i_p}$ terms in $ \Hm$.\\
      \lstinline|Output|&a list of syzygies $S^E_{i}\in\RuX^p$   \\
                         &such that the  $S^E_{i}$'s generate $\Syz(a_1M_1,\dots,a_pM_p)$.
    \end{tabular}}
\]

In the algorithm, we construct the syzygies  $S^E_{i}$ by successive concatenations of the lists obtained by the previous algorithm.

\begin{syzalgorithm}[syzygies of terms, see Definition~\ref{defisyz-terms0}, Discrete coherent context~\ref{discohcontext}]\label{genSalg2}
  \leavevmode
\begin{lstlisting}
SyzygiesOfTerms($a_1M_1,\dots,a_pM_p$)
local variables  $E:\text{subset of }\lrb{1\twodots p}$,  $S^E: \List(\RuX^p)$;
$S\gets $;
for $E$ in $\Pap(M_1,\dots,M_p)$ do
  $S^E \gets$ BasicSyzygiesOfTerms($a_1M_1,\dots,a_pM_p,E$);
  $S\gets S , S^E$
od;
return $S$
\end{lstlisting}
\end{syzalgorithm}

In the case of an ideal ($m=1$), one may forget about the basis vectors $e_{i_1},\dots,e_{i_p}$ and one has $\Pap(M_1,\dots,M_p)=\Pp$.

\subsection{S-list algorithms}\label{sec:slists-algor}

We have the following goal
corresponding to the S-list $\St(f_1,\dots,f_p)$ in Definition~\ref{defdef2} (Discrete \coh context~\ref{discohcontext}).
\[\makebox[\textwidth]{%
    \begin{tabular}{@{}rl@{}}
      \lstinline|Input| & $f_1,\dots,f_p \in\Hm$ not all zero,\\
      \lstinline|Output|&The S-list $\St=\St(f_1,\dots,f_p)$ of $f_1,\dots,f_p$ as in Definition~\ref{defdef2}.
    \end{tabular}}
\]

\begin{slistalgorithm}[S-list algorithm, Definition~\ref{defdef2}, Discrete \coh context~\ref{discohcontext}]\label{Slist1}

  \leavevmode
\begin{lstlisting}
Slist($f_1,\dots,f_p$)
local variables $Si,Ss:\RuX^p$; $S,Slist: \List(\RuX^p)$;
$\St \gets$;
$S\gets$ SyzygiesOfTerms($\LT(f_1,\dots,f_p)$);
for $Si$ in $S$ do
  $Ss\gets Si_{1}f_1+\dots+Si_{p}f_p$;
  if $Ss\neq 0$ then $\St\gets \St,Ss$ fi
od;
return $\St$
\end{lstlisting}
\end{slistalgorithm}

We have the following goal
corresponding to the S-list $\St^q(f_1,\dots,f_p)$ in Definition~\ref{defdef2} (Discrete \coh context~\ref{discohcontext}.)
\[\makebox[\textwidth]{%
    \begin{tabular}{@{}rl@{}}
      \lstinline|Input| & $q\in \NN$, $f_1,\dots,f_p \in\Hm$ not all zero,\\
      \lstinline|Output|&The iterated S-list $\St^q=\St^q(f_1,\dots,f_p)$ of $f_1,\dots,f_p$ as in Definition~\ref{defdef2}.
    \end{tabular}}
\]
\begin{slistalgorithm}[iterated S-list algorithm, Definition~\ref{defdef2}, Discrete \coh context~\ref{discohcontext}]\label{Slist2}\nopagebreak

  \nopagebreak\leavevmode\nopagebreak
\begin{lstlisting}
IteratedSlist($q,f_1,\dots,f_p$)
local variables $r:\NN$;
$\St^q\gets f_1,\dots,f_p$;
for $r$ from $1$ to $q$ do $\St^q\gets \St^q,$Slist($\St^q$) od;
return $\St^q$
\end{lstlisting}
\end{slistalgorithm}

\subsection{Rewriting algorithms}\label{sec:rewrite-algor}

The next algorithm corresponds to Proposition~\ref{syz-terms-2}.
\[\makebox[\textwidth]{%
    \begin{tabular}{@{}rl@{}}
      \lstinline|Input| & $f_1,\dots,f_p \in\Hm$ not all zero; $g_1,\dots,g_p$  $\in \RuX$; 
      \\  & the leading monomial~$\aLM$ of the summands of $\sum_{j=1}^p  g_jf_j=u$\\
                         &(let $E$ be the corresponding leading monomial index set) is $>\LT(u)$.\\
      \lstinline|Output|&$f_{1},\dots,f_{p+\ell}$ in $\St(f_1,\dots,f_p)$ extending $f_1,\dots,f_p$; $g_{1},\dots,g_{p+\ell}\in \RuX$ with\\
                        &the original $g_j$'s replaced by \(g_j-\LT(g_j)\) for $j\in E$ and unchanged outside $E$;\\
                        &$u=\som_{j=1}^{p+\ell} g_jf_j$  and the leading monomial of its summands is   $< \aLM$.
    \end{tabular}}
\]

\begin{rewritealgorithm}[rewriting a linear combination, Proposition~\ref{syz-terms-2}, Discrete \coh context~\ref{discohcontext}]\label{rewrite3}

  \leavevmode
\begin{lstlisting}
Rewriting($(g_1,f_1),\dots,(g_p,f_p)$)
local variables $j:\lrb{1\twodots p}$, $a_1,\dots,a_p,b_1,\dots,b_p,c_1,\dots,c_\ell:\R$,
                $M_1,\dots,M_p,\aLM, M^E:\Hm$, $N_1,\dots,N_p:\RuX$,
                $\ell,i:\NN$, $E:\text{subset of }\lrb{1\twodots p}$,
                $(S_{1,j}^E)_{j\in\lrb{1\twodots p}},\dots,(S_{\ell,j}^E)_{j\in\lrb{1\twodots p}}:\R[\uX]^p$, $s_1,\dots,s_\ell:\R^E$;
for $j$ in $\lrb{1\twodots p}$ do
  $a_j\gets \LC(f_j)$; $M_j\gets \LM(f_j)$; $b_j\gets \LC(g_j)$; $N_j\gets \LM(g_j)$ od;
$\aLM\gets\sup\sotq{N_jM_j}{j\in\lrb{1\twodots p}}$; $E\gets\sotq{j\in\lrb{1\twodots p}}{N_jM_j=\aLM}$;
$M^E\gets\lcm(M_j\mathrel;j \in E)$;
$(S_{1,j}^E)^{\vphantom{E}}_{j\in\lrb{1\twodots p}},\dots,(S_{\ell,j}^E)^{\vphantom{E}}_{j\in\lrb{1\twodots p}}\gets$BasicSyzygiesOfTerms($a_1M_1,\dots,a_pM_p,E$);
for $i$ in $\lrb{1\twodots \ell}$ do $s_i\gets (\LC(S_{i,j}^E))^{\vphantom{E}}_{j\in E}$ od;
find $(c_i)_{i\in\lrb{1\twodots \ell}}$ such that $(b_j)_{j\in E}=\sum_{i\in\lrb{1\twodots \ell}}c_i s_i$;
for $i$ in $\lrb{1\twodots \ell}$ do $g_{p+i}\gets c_i \frac \aLM{M^E}$; $f_{p+i}\gets  \som_{j \in E}   {S}_{i,j}^E f_j$ od;
for $j$ in $E$ do $g_j\gets g_j-\LT(g_j)$ od;
return $(g_1,f_1),\dots,(g_{p+\ell},f_{p+\ell})$
\end{lstlisting}
\end{rewritealgorithm}

We have the following goal
corresponding to Theorem~\ref{Grob-Buch} in Discrete \coh context~\ref{discohcontext}. This is an iteration of the previous one with a counter~$q$ for the number of iterations.
\[\makebox[\textwidth]{%
    \begin{tabular}{@{}rl@{}}
      \lstinline|Input|& $f_1,\dots,f_p \in \Hm\setminus\{0\}$, $g_1,\dots,g_p$  $\in \RuX$; we let $u=\sum_{j=1}^pg_jf_j$.\\
      \lstinline|Output|& $q \in\NN$ and $f_{1},\dots,f_{t}$ in $\St^q(f_1,\dots,f_p)$ extending $f_1,\dots,f_p$ such that\\
                       &$\LT(u)\in \gen{\LT(f_{1},\dots,f_{t})}$.
    \end{tabular}}
\]
\begin{rewritealgorithm}[iterated rewriting of a linear combination, Theorem~\ref{Grob-Buch}, Discrete \coh context~\ref{discohcontext}]\label{rewrite4}

  \leavevmode
\begin{lstlisting}
IteratedRewriting($(g_1,f_1),\dots,(g_p,f_p)$)
$q\gets 0$;
while $\LM(g_1f_1+\dots+g_pf_p)<\sup\{\LM(g_1)\LM(f_1),\dots,\LM(g_p)\LM(f_p)\}$ do
  $(g_1,f_1),\dots,(g_{p+\ell},f_{p+\ell})\gets$ Rewriting($(g_1,f_1),\dots,(g_p,f_p)$);
  $p\gets p+\ell$; $q\gets q+1$ od;
return $q,f_1,\dots,f_p$
\end{lstlisting}
\end{rewritealgorithm}

\section{Buchberger's algorithm}\label{sec:buchb-algor}

We have adjusted our framework so that the following theorem may be proved exactly as the analogue Theorem~4.2.3 in \citealt{AL}.

\begin{maintheorem}[\Sdcc~\ref{sdccontext}]
  \label{Grob-Buch-algo}
  \leavevmode
  \begin{enumerate}
  \item \emph{Buchberger's criterion}.
    Let \(f_1,\dots,f_p\in\Hm\) not all
    zero, and denote by \( S_1,\dots, S_{\ell} \)   the generators of \( \Syz(\LT(f_1,\dots,f_p))\)  computed in Proposition~\ref{syz-terms}. Then
    \(G=f_1,\dots,f_p\) is a Gröbner basis for  \(\gen{f_1,\dots,f_p}\) if  and only if for every \( 1 \leq i \leq \ell \), we have
    \[\overline{ S_{i,1}f_1+\dots+S_{i,p}f_p }^G=0.
    \]
  \item \emph{Buchberger's algorithm works}. Let $f_1,\dots,f_p \in \Hm\setminus \so0$ and $M=\gen{f_1,\dots,f_p}$. If the module of leading terms~$\MLT(M)$ of the module~$M$ is finitely generated, then the (generalised) Buchberger algorithm~\ref{genBuchbergeralg} computes a Gröbner basis for  \(\gen{f_1,\dots,f_p}\).
  \end{enumerate}
\end{maintheorem}

\smallskip
The (generalised) Buchberger algorithm has the following goal.
\[
  \begin{tabular}{@{}rl@{}}
    \lstinline|Input|&$f_1,\dots,f_p \in \Hm\setminus \so0$.\\
    \lstinline|Output|&a Gröbner basis $\lst{f_1,\dots,f_p,\dots,f_t}$ for $\gen{f_1,\dots,f_p}$ extending $\lst{f_1,\dots,f_p}$.
  \end{tabular}
\]

\begin{buchbalgorithm}[\Sdcc~\ref{sdccontext}]\label{genBuchbergeralg}
  \leavevmode

  \vbox{%
\begin{lstlisting}
Buchberger($f_1,\dots,f_p$)
local variables $\St: \List(\Hm)$; $f,r:\Hm$; $L:\List(\RuX)$;
$G\gets f_1,\dots,f_p$;
repeat
  $\St\gets$Slist($G$);
  for $f$ in $\St$ do
    remove $f$ from $\St$;
    $r, L\gets$ Division($f,G$);
    if $r\neq0$ then $\St\gets\St,r$ fi
  od;
  if $\St\neq \emptyset$ then $G\gets G,\St$ fi
until $\St= \emptyset$;
return $G$
\end{lstlisting}%
  }
\end{buchbalgorithm}

\begin{remark} \label{remBuchb1}
  If the algorithm terminates, then we can transform the obtained  Gröbner basis into a Gröbner basis $h_{1},\dots,h_{t'}$ such that no term of an element~$ h_{\ell}$ lies in $\gentq{\LT(h_{k})}{k\ne\ell}$ by replacing each element of the Gröbner basis with a remainder of it on division by the other nonzero elements and
  by repeating this process until it stabilises.  Such a Gröbner basis is called a \emph{pseudo-reduced} Gröbner basis. The terminology of ``reduced'' Gröbner basis is used only in the case where a way of normalising its elements is specified: see \citealt[Remark~239]{Y5}.
  \eoe
\end{remark}

\section{Schreyer's syzygy algorithm}\label{sec:Syzygy-algorithm}

\begin{definition}[\foreignlanguage{german}{Schreyer}'s monomial order]\label{defiGrob}
  Let $\R$ be a discrete ring. Consider a list $G=\lst{f_{1},\dots,f_{p}}$ in $\Hm\setminus \so0$ and the finitely generated submodule $U= \gen{G}=\R[\uX] f_{1}+\cdots+\R[\uX] f_{p}$  of $\Hm$.
  Let $(\epsilon_{1},\dots,\epsilon_{p})$ be the canonical basis of~$\R[\uX]^{p}$. \emph{\foreignlanguage{german}{Schreyer}'s monomial order induced by~\(>\) and \(G\)} on~$\R[\uX]^{p}$ is the order~\(>_G\) defined as follows:
  \[
    \uX^{\alpha}\epsilon_{k}>_G\uX^{\beta}\epsilon_{j}\quad\text{if}\quad\left|\,
      \begin{aligned}
        \text{either }&\LM (\uX^{\alpha} f_{k}) > \LM (\uX^{\beta} f_{j})\\
        \text{or both }&\LM (\uX^{\alpha} f_{k}) = \LM (\uX^{\beta} f_{j}) \text{ and $k<j$.}
      \end{aligned}\right.
  \]
\end{definition}

\foreignlanguage{german}{Schreyer}'s monomial order is defined on~$\R[\uX]^{p}$ in the same way as when~$\R$ is a discrete field \citep[see][p.~66]{EH}. Note that it actually depends only on
the leading monomials~\(\LM(G)\) of~\(G\).


Now we shall follow closely the ingenious proof by \citet{Sc} of Hilbert's syzygy theorem via Gröbner bases, but with  a strongly discrete coherent ring instead of a field. \foreignlanguage{german}{Schreyer}'s proof is very well explained in~\citealt[\S\S~4.4.1--4.4.3]{EH}.

Schreyer's syzygy algorithm below takes also place in \Sdcc~\ref{sdccontext} for $\R$.
It has the following goal.
\[
  \begin{tabular}{@{}rl@{}}
    \lstinline|Input|&a Gröbner basis $\lst{f_1,\dots,f_p}$ for a submodule of~$\Hm$.\\
   \lstinline|Output|
   &a Gröbner basis~$(u^E_i)_{1\leq i\leq\ell^E,E\in\Pap(\LM(f_1),\dots,\LM(f_p))}$ for $\Syz(f_1,\dots, f_p)$\\
         &with respect to Schreyer's monomial order induced by~$>$ and
         \(\lst{f_1,\dots,f_p}\).
 \end{tabular}
\]

\begin{Syzalgorithm}[\Sdcc~\ref{sdccontext}]\label{genSyzygyalg}%
  \leavevmode
\begin{lstlisting}
SchreyerSyzygy($f_1,\dots,f_p$)
local variables $(S_i^E)_{1\leq i\leq\ell^E,E\in\Pap(\LM(f_1),\dots,\LM(f_p))}:\List(\RuX^p)$,
                $q_1,\dots,q_p:\RuX$;
$(S_i^E)_{1\leq i\leq\ell^E,E\in\Pap(\LM(f_1),\dots,\LM(f_p))}\gets$ SyzygiesOfTerms($f_1,\dots,f_p$);
for $E$ in $\Pap(\LM(f_1),\dots,\LM(f_p))$ do
  for $i$ from $1$ to $l^E$ do
    $0,q_1,\dots,q_p\gets$Division($S^E_{i,1}f_1+\dots+S^E_{i,p}f_p,f_1,\dots,f_p$) (*@\label{genSyzygyalg:16}@*)
      (*@\textrm{(note that $\LM(S^E_{i,1}f_1+\dots+S^E_{i,p}f_p) \geq \LM(q_j f_j)$ whenever $q_j f_j\ne0)$}@*); (*@\label{genSyzygyalg:18}@*)
    $u^E_i\gets S^E_{i,1}\epsilon_1+\dots+S^E_{i,p}\epsilon_p-q_1\epsilon_1-\dots-q_p\epsilon_p$
  od       
od;(*@\label{gensyzygyalg:19}@*)
return $(u_i^E)_{1\leq i\leq\ell^E,E\in\Pap(\LM(f_1),\dots,\LM(f_p))}$
\end{lstlisting}
\end{Syzalgorithm}

Note that the polynomials $q_1,\dots,q_p$ of lines~\ref{genSyzygyalg:16}--\ref{genSyzygyalg:18} may have been computed while constructing the Gröbner basis.

\begin{remark} \label{remSyzgen}
For an arbitrary system of generators $\lst{h_{1},\dots,h_{p}}$ for a  submodule $U$ of~$\Hm$, the syzygy module of $\lst{h_{1},\dots,h_{p}}$ is easily obtained from the syzygy module of a Gröbner basis for $U$ \citep[see][Theorem~296]{Y5}.
\eoe
\end{remark}

\begin{maintheorem}[Schreyer's algorithm, Strongly discrete coherent context~\ref{sdccontext}]\label{schcoherent}
  Let \(U\) be a submodule of \(\Hm\) with Gröbner basis \(\lst{f_1,\dots,f_p}\).  Then  the relations~\(u^E_{i}\)  computed by \foreignlanguage{german}{Schreyer}'s syzygy algorithm~\ref{genSyzygyalg} form a Gröbner basis for the syzygy module \(\Syz(f_1,\dots, f_p)\) with respect to \foreignlanguage{german}{Schreyer}'s monomial order induced by~\(>\) and \(\lst{f_1,\dots,f_p}\). Moreover, for $E$ a position level subset of $\lrb{1\twodots p}$ and $1 \leq i \leq \ell^E$,
\begin{equation}\label{sch}
\LT (u^E_{i})=   {s}^E_{i,r}\ M^E/M_r \; \epsilon_{r} \hbox{ with } r = \min \sotq{j \in E}{{s}^E_{i,j} \neq  0}
\end{equation}
in the notation of Definition~\ref{defisyz-terms0} with \(M_1=\LM(f_1)\), \dots, \(M_p=\LM(f_p)\).
\end{maintheorem}

\begin{proof}[Proof \textup{(an analysis of the proof of~\citealt*[Theorem~4.16]{EH})}] Let us use the notation of \foreignlanguage{german}{Schreyer}'s sy\-zy\-gy algorithm~\ref{genSyzygyalg}. Recall that $u^E_{i}=(S^E_{i,1}-q_1)\epsilon_1+\dots+(S^E_{i,p}-q_p)\epsilon_p$, $S^E_{i,1}f_1+\dots+S^E_{i,p}f_p=q_{1}f_{1}+\dots+q_{p}f_{p}$,
  and $\LM(q_j f_j) \leq \LM(S^E_{i,1}f_1+\dots+S^E_{i,p}f_p) < (M^E/M_k)M_k  = M^E$ for any $k \in E$. So
$\LT (u^E_{i})= \LT (S^E_{i})=  {s}^E_{i,r}\ M^E/M_r \; \epsilon_{r} \hbox{ where } r = \min \sotq{j \in E}{{s}^E_{i,j} \neq  0}$.

\smallskip

\n Let us show now that the relations $u^E_{i}$ form a Gröbner basis for the syzygy module $\Syz(f_1,\dots, f_p)$. For this, let $v = v_1 \epsilon_1+\dots+v_p \epsilon_p \in \Syz(f_1,\dots, f_p)$ and let us show that  $\LT (v) \in \gentq{\LT (u^E_{i})}{1\leq i\leq\ell^E,E\in\Pap(M_1,\dots,M_p)}$.
Let us write $\LM (v_j \epsilon_j)=N_j \epsilon_j$ and $\LC (v_j \epsilon_j)=c_j$ for $1 \leq j \leq p$. Then $\LM (v)= N_{k} \epsilon_{k}$ for some $1 \leq k \leq p$. Now let $v' = \sum_{j\in D} c_j  N_j\epsilon_j$, where $D$ is the set of those $j$ for which $N_jM_j= N_{k} M_{k}$. By definition of \foreignlanguage{german}{Schreyer}'s monomial order, we have $j \geq k$ for all $j \in D$. Substituting each~$\epsilon_j$ in $v'$ by $T_j=\LT ({f_j})$, the sum becomes zero. Therefore $v'$ is a syzygy of the terms~$T_j$ with $j \in D$. By virtue of Proposition~\ref{syz-terms}, $v'$ is a linear combination of elements in \(S((T_j)_{j \in D}) \) of the form~$S^E_{i}$ with $E \subseteq D$ and $1 \leq i \leq \ell^E$. By inspecting the $j$th component of $v'$, we deduce that there exist $w_1,\dots,w_t \in \RuX$,  position level subsets $E_1,\dots,E_t$ of $D$
with $j \in E_1 \cap \cdots \cap E_t$, nonnegative integers $1 \leq i_1 \leq \ell^{E_1},\dots,1 \leq i_t \leq \ell^{E_t} $, such that
$c_jN_j= w_1 s^{E_1}_{i_1,j} M^{E_1}/M_j + \cdots+ w_t s^{E_t}_{i_t,j} M^{E_t}/M_j$, and $s^{E_1}_{i_1,j},\dots,s^{E_t}_{i_t,j} \neq 0$. As  $j > k$ for all $j \in D\setminus\{k\}$, it follows that $\LT (v') \in \gen{\LT(S^{E_1}_{i_1},\dots,S^{E_t}_{i_t})}$. The desired result follows since $\LT (v)=\LT (v')$.
\end{proof}

Schreyer's monomial order is a tailor-made term over position monomial order which changes
at each iteration, i.e.\ after each computation of a Gröbner basis of the syzygy module of the considered Gröbner basis. Schreyer's trick is, for $v = v_1 \epsilon_1+\dots+v_p \epsilon_p\in \Syz(f_1,\dots, f_p)$,
to prioritise (by deciding that they are greater) the
$\LM(v_j \epsilon_j)$ such that $\LM(v_j f_j)= \max (\LM(v_1 f_1),\dots, \LM(v_p f_p))$, and to order the obtained generators $u_i^E$ of $\Syz(f_1,\dots, f_p)$ in such a way that $X_n$ does not
appear in the leading terms of the $u_i^E$ (when computing a Gröbner basis for the first syzygy module), and to iterate this process in order to remove successively all the indeterminates $X_n,X_{n-1},\dots,X_1$ from the leading terms of the Gröbner basis of the
syzygy module. Once we reach this situation,
we continue the resolution over the base ring~$\R$.

As a consequence of Theorem~\ref{schcoherent}, we obtain the following constructive version of Hilbert's syzygy theorem for a strongly discrete coherent ring.

\begin{corollary}[Hilbert's syzygy theorem, Strongly discrete coherent context~\ref{sdccontext}]\label{hilbcoherent}
  \leavevmode
Let \(\Hm\) be a free \(\RuX\)-module with basis \((e_1,\dots,e_m)\), and \(>\) a monomial order on \(\Hm\). Let \(U\) be a finitely generated submodule of \(\Hm\) such that $\MLT(U)$ is finitely generated with respect to some monomial order. Then  \(M=\Hm/U\)  admits
an \(\RuX\)-resolution
\[
  0 \rightarrow F_q/V \rightarrow F_{q-1} \rightarrow \cdots \rightarrow F_1 \rightarrow F_0 \rightarrow M \rightarrow 0
\]

\n such that  \(q \leq n+1\), $F_0,\dots,F_q$ are finitely generated free \(\RuX\)-modules, and $V$ is generated by finitely many iterated syzygies whose leading terms with respect to Schreyer's induced monomial order do not depend on  the indeterminates  $X_n,\dots,X_1$.
\end{corollary}

\begin{proof}
Let $ (f_1,\dots,f_p)$ be a  Gröbner basis for $U$ with respect to the considered order. Reorder the $f_j$'s so that  whenever $\LM (f_j)$ and $\LM (f_k)$ involve the same position for
some $k<j$, say $\LM (f_k)=M'_k e_{i_k}$ and  $\LM (f_j)=M'_j e_{i_j}$ with \(i_j=i_k\), then $\deg_{X_n}(M'_k) \geq  \deg_{X_n}(M'_j)$. Consider the leading monomials $\LM(u^E_{i})=M^E/M_r\;\epsilon_{r}$, $r = \min \sotq{j \in E}{{s}^E_{i,j} \neq  0}$, computed by \foreignlanguage{german}{Schreyer}'s syzygy algorithm~\ref{genSyzygyalg}: our reordering entails that \(\deg_{X_n}M_E=\deg_{X_n}M_r\), so that the indeterminate~$X_n$ does not appear in~$\LM(u^E_{i})$. Thus, after at most $n$ computations of the iterated syzygies, we reach the desired situation.
\end{proof}

\begin{remark}
  This theorem generalises Theorems~5.5 and~6.2 of \citealt*{GLNY2020}. Note that their statement there, as well as that of its Theorem~5.9, needs to be amended: ``the TOP lexicographic monomial order'' needs to be replaced by ``some monomial order''; the proof given here shows how their proof needs to be amended. See the arXiv version \citealt*{GLNY2020arxiv}.
\end{remark}

\addcontentsline{toc}{section}{References}

\end{document}